\documentclass[11pt]{article}
    \usepackage{amsmath}
    \usepackage{amsfonts}
    \usepackage{amssymb}
    \usepackage{amsthm}
    \usepackage{mathrsfs}
    \usepackage{lscape}
    \usepackage{euscript}
    \usepackage{latexsym,enumerate,color,mdwlist}
    \usepackage{tikz}
    \usepackage{a4wide}
    \usepackage{hyperref}
    \usepackage{soul}
    \usepackage{mathtools}
    \usepackage[shortlabels]{enumitem}

    \usepackage{apptools}
    \setlist[enumerate]{leftmargin=25pt}
    \setlist[itemize]{leftmargin=25pt}

    \theoremstyle{plain}
    \newtheorem{Thm}{Theorem}[section]
    \newtheorem{prop}[Thm]{Proposition}
    \newtheorem{lem}[Thm]{Lemma}
    
    \newtheorem{cor}[Thm]{Corollary}

    \theoremstyle{definition}
    \newtheorem{de}[Thm]{Definition}

    \newtheorem{con}{Conjecture}

    \newtheorem{summary}{Summary}
    \newtheorem*{theorem*}{Theorem}
    \theoremstyle{remark}

    \def\N{\mathbb{N}}
    \def\Q{\mathbb{Q}}
    \def\Z{\mathbb{Z}}
    \def\C{\mathbb{C}}
    \def\R{\mathbb{R}}
    \def\La{\Lambda}
    \def\la{\lambda}
    
    \def\stb{,\ldots ,}

    \newcommand{\bra}[1]{{\left({#1}\right)}}
    \newcommand{\set}[1]{{\left\{{#1}\right\}}}

    \DeclarePairedDelimiter\abs{\lvert}{\rvert}

    \newcommand{\scal}[1]{{\left\langle{#1}\right\rangle}}

    \newcommand{\ze}{\ensuremath{\zeta}}
    \newcommand{\Om}{\ensuremath{\Omega}}

    \newcommand{\vn}{\ensuremath{\varnothing}}
    \newcommand{\ssq}{\ensuremath{\subseteq}}
    \newcommand{\znp}{\ensuremath{\Z_N^{\star}}}
    \newcommand{\FF}{\ensuremath{\mathbb{F}}}

    \newcommand{\bs}[1]{\ensuremath{\mathbf{#1}}}
    \newcommand{\ol}{\ensuremath{\overline}}

\makeatletter
\newcommand{\subjclass}[2][1991]{%
  \let\@oldtitle\@title%
  \gdef\@title{\@oldtitle\footnotetext{#1 \emph{MSC classes.} #2}}%
}
\newcommand{\keywords}[1]{%
  \let\@@oldtitle\@title%
  \gdef\@title{\@@oldtitle\footnotetext{\emph{Keywords:} #1.}}%
}
\makeatother

\begin{document}
\title{Fuglede's conjecture holds for cyclic groups of order $pqrs$}
\author{Gergely Kiss \thanks{Alfr\'ed R\'enyi Institute of
Mathematics. e-mail: kigergo57@gmail.com}
\qquad Romanos Diogenes Malikiosis \thanks{Aristotle University of
Thessaloniki, Department of Mathematics. e-mail: rwmanos@gmail.com}
\qquad G\'abor Somlai \thanks{E\"otv\"os Lor\'and University, Faculty of Science, Institute of Mathematics. e-mail: gsomlai@cs.elte.hu}
\qquad M\'at\'e Vizer \thanks{Alfr\'ed R\'enyi Institute of
Mathematics. e-mail: vizermate@gmail.com.}}
\date{}


\maketitle
\begin{abstract}The tile-spectral direction of the discrete Fuglede-conjecture is well-known for cyclic groups of square-free order, initiated by \L aba and Meyerowitz, but the spectral-tile direction is far from being well-understood. The product of at most three primes as the order of the cyclic group was studied intensely in the last couple of years.

In this paper we study the case when the order of the cyclic group is the product of four different primes and prove that Fuglede's conjecture holds in this case. \end{abstract}

\section{Introduction}
The conjecture of Fuglede \cite{F1974} was originated from a problem posed to him by Segal on the restriction of the usual differential operators
$\frac{\partial}{\partial x_i}$ (acting in the distribution sense) to commuting self-adjoint operators on $L^2(\Om)$, where $\Om$ is a bounded measurable subset of $\R^n$ with
$0<m(\Om)<+\infty$.
Fuglede showed that this happens
if and only if $\Om$ accepts an orthogonal basis of complex exponentials $e^{2\pi i\scal{\la,x}}$, which is the definition of a \emph{spectral} set. The set of frequencies $\la$
denoted by $\La$, is called the \emph{spectrum} of $\Om$.

We say that $S$ is a \textit{tile} of
$\mathbb{R}^{n}$, if there is a set $T \subset \mathbb{R}^{n}$ such
that almost every point of $\mathbb{R}^{n}$ can be uniquely written
as $s + t$, where $s \in S$ and $t \in T$. In this case, we say
that $T$ is the \textit{tiling complement of $S$}.

In the same paper, Fuglede proceeded to state the so-called Spectral Set Conjecture (that we will just call Fuglede's
conjecture):

\begin{con}\label{Fuglede}
$\Omega$ is spectral if and only if $\Omega$ is a tile.
\end{con}

Fuglede proved this conjecture when the spectrum or the tiling complement is a lattice. However, this conjecture was largely proven to be false, initially by Tao \cite{T2004} in 2004
for dimensions $n\geq 5$, and subsequently for $n\geq3$ \cite{FMM2006,KM2006}. In a recent paper \cite{LM2019} Fuglede's conjecture has verified for all convex bodies in all dimensions. The status of this conjecture remains open for $\R$ and $\R^2$. Recently, it was shown that a spectral subset of $\R$ of Lebesgue measure $1$ has a spectrum that is
periodic \cite{IK2013}.

The notable feature of Tao's proof was the transition to the setting of finite abelian groups (although this had begun with \L{}aba \cite{L2001}, connecting Fuglede's conjecture with
the tiling results of Coven and Meyerowitz \cite{CM1999}).
Let us now state Fuglede's conjecture for abelian groups in a more precise way.

\begin{de}
Let $\Z_N$ denote the cyclic group of order $N$.

\smallskip

$\bullet$ $S\subset\Z_N$ is called \emph{spectral}, if there is some $\La\subset\Z_N$ with $\abs{S}=\abs{\La}$ and the exponential functions $e_{\la}(x)=\xi_N^{\la x}$ form an orthogonal basis over $S$, that is
 \begin{equation}\label{orthogonality}
  \scal{e_{\la},e_{\la'}}_S:=\sum_{s\in S}e_{\la}(s)\ol{e_{\la'}(s)}=\abs{S}\delta_{\la\la'}
 \end{equation}
for every $\la,\la'\in\La$, where $\xi_N=\exp(2\pi i/N)$ a primitive $N$'th root of unity. We call $\La$ the \emph{spectrum} and $(S,\La)$ is called a \emph{spectral pair} of $\Z_N$.

\smallskip

$\bullet$ $S \subset \Z_N$ is called a \emph{tile}, if there is another subset $T\subset\Z_N$ such that every element of $\Z_N$ can  uniquely be written
 as $s+t$, where $s\in S$, $t\in T$. We call $T$ the \emph{tiling complement} of $S$, we write $S\oplus T $ \footnote{In general, we write $A\oplus B$ if every element in the sumset $A+B$ can be written uniquely as a sum of an element of $A$ and an element of $B$.} $=\Z_N$  and we call $(S,T)$ a \emph{tiling pair} of $\Z_N$.

\end{de}

\begin{con}\label{fugconcyc} For any $N$ and $S \subset \mathbb{Z}_N$ we have that $S$ is spectral if and only if $S$ tiles $\Z_N$.
\end{con}

Borrowing the notation from \cite{DL2014} and \cite{MK17}, we write
$\mathbf{S-T}(G)$ (resp. $\mathbf{T-S}(G)$), if the $Spectral
\Rightarrow Tile$ (resp. $Tile \Rightarrow
Spectral$) direction of Fuglede's conjecture holds in $G$. The above mentioned connection between the conjecture on
$\mathbb{R}$, on $\mathbb{Z}$ and on finite cyclic groups is
summarized below \cite{DL2014} (where $\mathbf{T-S}(\mathbb{Z}_{\mathbb{N}})$ means that $\mathbf{T-S}(\mathbb{Z}_n)$ holds for every $n \in \N$):

\
$$\mathbf{T-S}(\mathbb{R}) \Longleftrightarrow \mathbf{T-S}(\mathbb{Z})
\Longleftrightarrow \mathbf{T-S}(\mathbb{Z}_{\mathbb{N}}),$$

$$\mathbf{S-T}(\mathbb{R}) \Longrightarrow \mathbf{S-T}(\mathbb{Z})
\Longrightarrow \mathbf{S-T}(\mathbb{Z}_{\mathbb{N}}).$$ \

This close connection shows the importance of Fuglede's Conjecture for abelian groups. Since the result of Tao appeared, there have been many results on finite abelian groups by many authors, which we summarize below:

\begin{Thm}\label{discreteFuglederesults}
 Let $G$ be a finite abelian group and $p,q,r$ different primes.
 \begin{enumerate}[(1)]
  \item If $G=\Z_p^d$ where $p$ is an odd prime and $d\geq4$, then there is a spectral subset of $G$ that does not tile \cite{Aten2016,FS2020}.
  \item If $G=\Z_2^d$ where $d\geq10$, then there is a spectral subset of $G$ that does not tile \cite{FS2020}.
  \item If $G=\Z_8^3$, then there is a spectral subset of $G$ that does not tile \cite{KM2006}.
  \item If $G=\Z_{24}^3$, then there is a tile of $G$ that is not spectral \cite{FMM2006}.
  \item If $G=\Z_p^d$ where $p$ is prime and $d\leq3$, then any tile of $G$ is spectral \cite{Aten2016}. The converse holds if $d\leq2$ \cite{IMP2017} or $d=3$ and $p\leq7$.
  \item If $G=\Z_p\oplus\Z_{p^2}$ \cite{Shi19} or $G=\Z_p\oplus\Z_{pq}$ \cite{KS2019}, where $p\neq q$ are primes, then a subset of $G$ is spectral if and only if it is a tile.
  \item If $G=\Z_N$ and $N$ is square-free or $N=p^mq^n$ or $N=p^n d$, where $d$ is square-free, then any tile of $G$ is spectral \cite{L2002,MK20, Shi18}.
  \item\label{item8} If $G=\Z_N$ and $N=p^n, p^nq, p^nq^2, pqr, p^2qr$, then a subset of $G$ is spectral if and only if it is a tile \cite{KMSV2018,L2002,MK17, MK20, Shi18,So2019}.
 \end{enumerate}
\end{Thm}

Using the fundamental theorem of finite abelian groups, the following holds for $G$ whose minimum number of generators is $d$:
\[G\cong\Z_{n_1}\oplus\Z_{n_2}\oplus\dotsb\oplus\Z_{n_d},\]
where $n_i\mid n_{i+1}$ for $1\leq i\leq d-1$. Thus, we obtain

\begin{cor}
 Let $G$ be a finite Abelian group, and $d$ be the minimum number of generators of $G$.
 \begin{itemize}
  \item If $d\geq4$ and $\abs{G}$ is odd, then there is a spectral subset of $G$ that does not tile.
  \item If $d\geq10$, then there is a spectral subset of $G$ that does not tile.
 \end{itemize}
\end{cor}

In this article we investigate the discrete Fuglede's conjecture for cyclic groups of order of the product of 4 different primes.  In particular we prove
\begin{Thm}\label{main}
Fuglede's conjecture is true for $\Z_{pqrs}$, where $p,q,r,s$ are different primes.
\end{Thm}

\noindent
Combining the results in Theorem \ref{discreteFuglederesults} \ref{item8} and Theorem \ref{main} we obtain the following corollary.
\begin{cor}\label{cmain}
Fuglede's conjecture is true for $\Z_{pqrs}$, where $p,q,r,s$ are arbitrary primes.
\end{cor}

\section{Preliminaries; Structures and tools}

\subsection{Basic properties of Spectral and Tiling pairs}

In this section we list some properties of spectral pairs that we use during the proof. Our first remark is that we will always identify the dual group of a cyclic group $\Z_N$ with itself.

\

$\bullet$ First note that spectral pairs satisfy the following duality property.

\begin{lem}\label{duality}
Let $G$ be a finite abelian group. Assume that $S \subset G$ is a spectral set having $\Lambda$ as a spectrum. Then $S$ is also a spectrum for $\Lambda$.
\end{lem}

\

$\bullet$ A trivial property of a tiling pair $(T,S)$ is $(T-T)\cap(S-S)=\set{0}$.
There is a similar property for spectral pairs involving the difference set of a spectrum that we introduce now. Let $(S,\La)$ be a spectral pair of $\Z_N$. Let us denote by $\bs1_S$ the characteristic function of $S$, and the Fourier transform of a function $f:\Z_N\to\C$ as
\[\hat{f}(y)=\sum_{x\in\Z_N}f(x)\xi_N^{-xy},\]
we have for any $\la,\la'\in\La$, where $\la\neq\la'$,
\[0=\sum_{s\in S}e_{\la}(s)\ol{e_{\la'}(s)}=\sum_{s\in S}\xi_N^{(\la-\la')s}=\hat{\bs1}_S(\la'-\la)\]
by \eqref{orthogonality}, therefore
\begin{equation}\label{spectrality}
 \La-\La\subset\set{0}\cup\set{d\in\Z_N:\hat{\bs1}_S(d)=0},
\end{equation}
that is, the difference set of a spectrum is a subset of the zero set of the Fourier transform of $\bs1_S$ along with $\set{0}$.

\

$\bullet$ Relation \eqref{spectrality} can also be expressed with the notion of the mask polynomial of a multiset.

\begin{de}\label{defmask} For a cyclic group $G$, a field $K$ and a function $f:G \rightarrow K$, we define the \textit{mask polynomial of} $f$ (in $K[x]$) by $$m_{f}(x)=\sum_{g \in G}f(g) x^g.$$

For a (multi)set $S$ on $G$ we denote  the multiplicity function of $S$ by $f_S$ (that is a function from $G$ to $\N$) and we denote by $m_S$ the mask polynomial of $f_S$.
\end{de}

The basic connection between the mask polynomial and the Fourier transform of (the characteristic function of) a set $S$ is
\[\hat{\bs1}_S(d)=m_S(\xi_N^{-d}),\]
so we may rewrite \eqref{spectrality} as
\begin{equation}\label{spectralitymask}
 \La-\La\subset\set{0}\cup\set{d\in\Z_N:m_S(\xi_N^{-d})=0}.
\end{equation}

We define $Z(S):=\set{d\in\Z_N:m_S(\xi_N^{-d})=0}$. Since $m_S(x)\in\Z[x]$, the following holds: if $g\in\Z_N^{\star}$, then
\[m_S(\xi_N^d)=0\Longrightarrow m_S(\xi_N^{dg})=0.\]
It is clear from \eqref{spectralitymask} that $Z(S)$ is a union of subsets of the form $d\Z_N^{\star}=\set{dg:g\in\Z_N^{\star}}$, where $d\mid N$, in particular for a spectral pair $(S,\La)$ in $\Z_N$ we have
\begin{equation}\label{divclass}
\La-\La\ssq\set{0}\cup\bigcup_{d\mid N,\	m_S(\xi_N^d)=0}d\Z_N^{\star}.
\end{equation}

\subsection{Geometric interpretation of cyclic groups of square-free order and the cube rule}

$\bullet$ For $m,n \in \N$ with $m \mid  n$ let us define the \textit{natural projection} of $\Z_n$ onto $\Z_m \le \Z_n$, $\mathbb{Z}_m$ denotes the only subgroup of $\mathbb{Z}_n$ of order $m$.
Let $k \in \N$ with $\frac{n}{m} \mid k$ and $k \equiv 1 \pmod{m}$. Then $x \to k\cdot x$ defines a surjective homomorphism from $\Z_n$ to $\Z_m$.
Note that the function defined in this way is independent on the choice of $k$. If $U$ is a subset of $\Z_n$, then its \textit{projection} to $\Z_m$, which we denote by $U_m$, is a multiset defined as $\{k\cdot u \mid u \in U \}$.

\

In this subsection our goal will be to introduce a geometric interpretation of cyclic groups of square-free order and an important observation that we call the cube rule. We use the  `geometric language' introduced below throughout the proof.

\

$\bullet$ Note that the cyclic group $\Z_{N}$ with $N =p_1 p_2 \dots p_k$ ($p_i'$s are different prime numbers) can be written as the direct sum $ \oplus_{i=1}^k \Z_{p_i}$, so its elements can also be considered as $k$-tuples. Let the \textit{Hamming distance} of two elements $x, y \in \oplus_{i=1}^k \Z_{p_i}$ be the number of coordinates they differ. Let us denote by $d_H(x,y)$ the Hamming distance of $x$ and $y$.

\noindent
We consider $\oplus_{i=1}^k \Z_{p_i}$ as the subset of the $k$-dimensional integer grid, so we identify it with
\[\{ x=(x_1,x_2, \ldots ,x_k) \in \Z^k \mid 0 \le x_1 \le p_1-1, \ 0 \le x_2 \le p_2-1, \  \ldots, \ 0 \le x_k \le p_k-1\}.
\]

\noindent
In this interpretation every coset of a subgroup of order $d=p_{i_1}\dots p_{i_s}$ of $\Z_N$ can be identified with a proper $s$-dimensional affine subspace of
$\oplus_{i=1}^k \Z_{p_i}$.
For instance, $\Z_{p_1}$-cosets are lines that are the set of vertices with the same $\Z_{p_2}$-, $\ldots, \Z_{p_k}$-coordinates.

\

$\bullet$ Now we recall an important tool that was introduced in \cite{KMSV2018}. It describes a condition that a multiset $B \subseteq \Z_N$ should satisfy, in case $N$ is square-free and $ \Phi_N \mid m_B$. For an integer $N$ we denote by $\omega(N)$ the number of (different) prime divisors of $N$. If $N$ is square-free, then we can consider $\Z_N=\oplus_{i=1}^{\omega(N)}\Z_{p_i}$ as the direct sum of groups of different prime order. We call a subset $C$ of $\Z_N$ an \textit{$\omega(N)$-dimensional cube}, if $C=  \oplus_{i=1}^{\omega(N)} A_i$, where $A_i \subset \Z_{p_i}$ are $2$-element sets and let $c_0 \in C$.

\begin{Thm}\label{cuberule}{(\textbf{Cube rule}; \cite{KMSV2018}, Proposition 3.5.)} Suppose that $N$ is a square-free integer, $B \subseteq \Z_{N}$ is a multiset and  $ \Phi_N \mid m_B$. Then for every $\omega(N)$-dimensional cube $C$ and $c_0\in C$ the following holds
 \begin{equation}\label{eqcuberule}
     \sum_{c \in C}(-1)^{d_H(c_0,c)}B(c)=0,
 \end{equation}
where $B(c)$ denotes the multiplicity of $c$ in $B$.
\end{Thm}

We say that $B$ satisfies the \emph{$\omega(N)$-dimensional cube rule}, if equation \eqref{eqcuberule} holds for every  $\omega(N)$-dimensional cube $C$. For shortening our notation a 2-, 3- or 4-dimensional cube (rule) will be denoted by 2D, 3D, 4D cube (rule), respectively.

\

$\bullet$
In general, for square-free $N=\prod_{i=1}^{\omega(N)}p_i$ we may introduce the $d$-dimensional cube $C_{p_{i_1}\dots p_{i_d}}= \oplus_{j=1}^{d} A_{p_j}$, where $A_{p_j} \subset \Z_{p_j}$ are proper $2$-element sets.
If $x,y\in \Z_N$ with $d_H(x,y)=d$, then we denote by $C_{p_{i_1} \dots p_{i_d}}(x,y)$ the unique $d$-dimensional cube determined by $x,y$, where $x,y$ differ in $p_{i_1}, \dots, p_{i_d}$-coordinates.
If it is clear form the context, then we simply denote this cube by $C(x,y)$.

\subsection{Preliminary statements}

Let us start with some statements from \cite{KMSV2018}.

\begin{lem}\label{zero}
Let $G$ be a finite abelian group. It is enough to prove the Spectral $\Rightarrow$ Tile direction of Fuglede's conjecture for  spectral pairs $(S, \La)$
with $0 \in S$ and $0 \in \La$.
\end{lem}

\begin{lem}\label{propgenerate}
Let $G$ be a finite abelian group and $S$ a spectral set in $G$,
that does not generate $G$. Assume that for every proper subgroup
$H$ of $G$ we have $\mathbf{S-T}(H)$. Then $S$ tiles $G$.
\end{lem}

\begin{lem}\label{lem1} Let $N$ be a natural number and suppose that $\mathbf{S-T}(\Z_{N}/H)$ holds for every $1 \neq H\le \Z_N$.
Assume that  $(S, \La)$ is a spectral pair and $\La$ is not primitive.
Then $S$ tiles $\Z_{N}$.
\end{lem}

\begin{lem}\label{propcsakzpeltolt}
Let $N$ be a natural number, $S$ a spectral set in $\Z_N$ and $p$ a prime divisor
of $N$. Assume that
$\mathbf{S-T}(\Z_{\frac{N}{p}})$. If $S$ is the union of
$\mathbb{Z}_p$-cosets, then $S$ tiles $\mathbb{Z}_{N}$.
\end{lem}

\begin{lem}\label{lazpcos}
Let $(S, \La)$ be a spectral pair in $\Z_N$, where $N=p \cdot p_1\dotsm p_k$ is a square-free integer
and ${\bf S-T}(\Z_{\frac{N}{p}})$ holds. Suppose that $\La$ is the union of $\Z_p$-cosets. Then $S$ is a tile.
\end{lem}

\begin{proof}
Let $\Lambda':= \Lambda \cap$  $ \Z_{\frac{N}{p}}$.
The mask polynomial $m_{\Lambda'}$ has degree at most $\frac{N}{p}-1$ and $m_{\Lambda'}(1)=|\Lambda'|= \frac{|\Lambda|}{p}$ as $\La$ is the union of $\Z_p$-cosets. Note that the mask polynomial of the $\Z_p$-coset containing 0 is $\Phi_p(x^{\frac{N}{p}})$. Since $\Lambda$ is the union of $\Z_p$-cosets we have
$\Phi_p(x^{\frac{N}{p}}) \mid m_{\Lambda}$. Moreover $\La = \La' \oplus \Z_p$ implies
\begin{equation}\label{eqla'}
    m_{\Lambda}=m_{\Lambda'}\Phi_p(x^{\frac{N}{p}}).
\end{equation}

Since $\Lambda$ is the union of $\Z_p$-cosets, we have $\Phi_p \mid m_S$, so $S$ is equidistributed on the $\Z_{\frac{N}{p}}$-cosets.
Note that, if $|S|=|\La|=p$, then every $\Z_{\frac{N}{p}}$-coset contains exactly one element of $S$, hence $S$ is a tile. Thus, we can assume that $|\La|=|S|>p$.
Let $S_i=\{x \in S \mid x \equiv i \pmod{p} \}$. Hence $|S_i|=\frac{|S|}{p}=\frac{|\La|}{p}=|\Lambda'|>2$ holds.

As $(\La,S)$ is a spectral pair this implies that  for any $s_1, s_2\in \Z_N$ we get that
$\Phi_m\mid m_{\La}$, where $m$
is the order of $s_1-s_2$.
Since $N$ is square-free, we have $p\nmid m$ for $s_1-s_2\in S_i$, $s_1\ne s_2$ and $1\le i\le n$. This implies that $\Phi_m\nmid\Phi_p(x^{\frac{N}{p}})$ hence, by \eqref{eqla'}, we obtain $\Phi_m\mid m_{\La'}$.
By $|S_i|=|\La'|$, it follows that $S_i-i$ is a spectrum of $\La'$ for every $0 \le i \le p-1$.

Then by assumption we have that $S_i-i$ is a tile on $\Z_{\frac{N}{p}}$. As $N$ is square-free, by Proposition 4.1 of \cite{Shi18}, we have that $S_i-i$ is a set of coset representatives of
$\Z_{\frac{N}{p|S_i|}}$ in $\Z_{\frac{N}{p}}$.
Therefore, the subgroup $T=\Z_{\frac{N}{|S|}}$ is a common tiling complement of $S_i-i$ in $\Z_{\frac{N}{p}}$ for each $S_i$, hence it is a tiling complement for $S$ in $\Z_N$, finishing the proof of the lemma.

\end{proof}
\noindent
By \cite{Shi18}, we have that ${\bf S-T}(\Z_{pqr})$ holds, where $p,q,r$ are different primes. Combining with Lemma \ref{lazpcos} we get the following.

\begin{cor}
Let $(S, \La)$ be a spectral pair in $\Z_N$, where $N=pqrs$ with $p,q,r,s$ different prime numbers. Suppose that $\La$ is the union of $\Z_p$-cosets. Then $S$ is a tile.
\end{cor}

\begin{lem}\label{manyprimes}
Let $N=p_1p_2\dotsm p_k$ be a square-free integer, and let $S\ssq\Z_N$ be a spectral set whose cardinality
is divisible by at least $k-1$
primes among $p_1,\dotsc,p_k$. Then $S$ tiles $\Z_N$.
\end{lem}

\begin{proof}
If $N\mid\abs{S}$, then $S=\Z_N$ and there is nothing to prove. So, without loss of generality
we can suppose that $p_1\dots p_{k-1}\mid\abs{S}$ and $p_k\nmid\abs{S}$. Let $\La\ssq\Z_N$ be a spectrum of $S$. Since
$p_k\nmid |S|=\abs{\La}$, we must have $m_{\La}(\ze_{p_k})\neq0$, therefore \eqref{spectralitymask} yields
\[(S-S)\cap\frac{N}{p_k}\Z_N=\set{0}.\]
Thus, every element of $S$ is unique $\bmod ~\frac{N}{p_k}$, yielding $\abs{S}\leq p_1\dotsm p_{k-1}$, and combined
with $p_1\dotsm p_{k-1}\mid\abs{S}$ we finally obtain
\[\abs{S}=p_1\dotsm p_{k-1}.\]
Therefore, $S$ consists of a complete set of coset representatives modulo the subgroup $\frac{N}{p_k}\Z_N$, and so this
subgroup is exactly a tiling complement of $S$.
\end{proof}

\noindent
Lemma \ref{manyprimes} immediately implies the following.

\begin{cor} \label{rsndiv}
Let $N=pqrs$ with different primes $p,q,r,s$ and $(S, \La)$ a spectral pair of $\Z_N$. Then, in order to prove ${\bf S - T}( \Z_{pqrs})$, we can assume that $r\nmid |S|$ and $s\nmid |S|$ (otherwise we are done by Lemma \ref{manyprimes}). This automatically implies that $\Phi_r, \Phi_s\nmid m_S, m_{\La}$.
\end{cor}

\noindent
We also remind that if the spectral set is small, then it tiles.

\begin{Thm}[\cite
{KM2006}, Theorem 2.1]\label{T6}
Let $S$ be a spectral set in a finite abelian group $G$ with $\abs{S}\leq5$. Then, $S$ tiles $G$.
\end{Thm}

\noindent
Let us gather the information above in the following summary.

\begin{summary}\label{summary} Let $N=p_1\ldots p_k$ be a square-free integer and $(S, \La)$ a spectral pair of $\Z_N$.
Any of the following assertions implies that $S$ tiles $\Z_N$.

\medskip

$\bullet$ $S$ or $\La$ does not generate $\Z_N$.

\smallskip

$\bullet$ $S$ or $\La$ is the union of $\Z_{p_i}$-cosets for some $p_i$.

\smallskip

$\bullet$ all $p_i'$s except at most one divide the cardinality of $S$.

\smallskip

$\bullet$ $|S| \le 5$.

\medskip

\noindent  We can also suppose that $0 \in S$ and $0 \in \La$ to prove the Spectral $\Rightarrow$ Tile direction of Fuglede's conjecture.

\end{summary}

\subsection{Technical lemmata for the proof}

\
$\bullet$ Let $N=p_1 \ldots p_k$ be a square-free natural number. For any primitive $N$'th root of unity $\xi_N$ and for an integer $t$ one can uniquely write the complex number $\xi_N^t$ as $\xi_{p_1}^{a_1(t)} \cdot \xi_{p_2}^{a_2(t)} \cdot \ldots \cdot \xi_{p_k}^{a_k(t)},$ where $0 \le a_i(t) \le p_i-1$, for $i=1 \stb k$.

\begin{lem}\label{lemprop1}
Let $N=p_1 \ldots p_k$ be a square-free integer and $B\subset \Z_{N}$ a multiset. Assume $\Phi_{\frac{N}{p_1}} \mid m_B$ and $\Phi_{N} \mid m_B$. Let us denote by $B_j=\{ b \in B  \mid b \equiv j \pmod{p_1} \}$. Consider $B_j$ as a subset of $\Z_{\frac{N}{p_1}}$ by simply identifying $B_j$ with $B_j-j$. Then $\Phi_{\frac{N}{p_1}} \mid m_{B_j}$ for every $j=0 \stb p-1$.
\end{lem}
\begin{proof}

Note that $\Phi_{N} \mid m_B$  implies
\[ 0= \sum_{t \in B} \xi_{p_1}^{a_1(t)}
\ldots
\xi_{p_k}^{a_k(t)}.
\]
Let $\ell$ be an integer with $(\ell,p_1)=1$ and $\ell p_1 \equiv 1 \pmod{\frac{N}{p_1}}$. Then $\xi_N^{\ell p_1}$ is a primitive $\frac{N}{p_1}$'th root of unity and

\begin{equation}\label{eq11}
     \xi_N^{(\ell p_1)t}=\xi_{p_1}^{(\ell p_1)a_1(t)}\xi_{p_2}^{(\ell p_1)a_2(t)} \ldots \xi_{p_k}^{(\ell p_1)a_k(t)}=\xi_{p_2}^{a_2(t)}\ldots \xi_{p_k}^{a_k(t)}.
\end{equation}
Since $\xi_N^{\ell p_1}$ is a primitive $\frac{N}{p_1}$'th root of unity, by equation \eqref{eq11} we get that  $\Phi_{\frac{N}{p_1}} \mid m_B$ is equivalent to
\[ 0=  \sum_{t \in B} \xi_{p_2}^{a_2(t)}
\ldots
\xi_{p_k}^{a_k(t)}, \]
where $a_i(t) \equiv t \pmod{p_i}$, for $i=2 \stb k$.

\smallskip

\noindent
Using our notion $B_j$ we can rewrite these equations as
\[ 0= \sum_{j=0 }^{p_1-1} \xi_{p_1}^{j} \sum_{t \in B_{j}}\xi_{p_2}^{a_2(t)}\ldots \xi_{p_k}^{a_k(t)},
\]
and
\[ 0= \sum_{j=0 }^{p_1-1}  \sum_{t \in B_{j}}\xi_{p_2}^{a_2(t)}\ldots \xi_{p_k}^{a_k(t)},
\]
Let $t_j=\sum_{t \in B_{j}}\xi_{p_2}^{a_2(t)}\ldots \xi_{p_k}^{a_k(t)}$. Thus $0=\sum_{j=0 }^{p_1-1} \xi_{p_1}^{j} t_j$ and $0=\sum_{j=0 }^{p_1-1}  t_j$. Since $t_j \in \Q(\xi_{\frac{N}{p_1}})$ and the    minimal polynomial of $\xi_{p_1}$ over $\Q(\xi_{\frac{N}{p_1}})$ is $1+x+\ldots+ x^{p_1-1}$,
it follows that $t_j$'s are constant. Using $\sum_{j=0}^{p_1-1}t_j=0$ we obtain $t_j=0$  for every $j=0 \stb p_1-1$. \end{proof}

Note that Lemma \ref{lemprop1} can be generalised as it is stated in Proposition \ref{pprop}. As we do not use it in the proof of the main result we present it in Appendix B.

\begin{lem}\label{l1}
Let $N=pqrs$ for different prime numbers $p,q,r,s$ and $S \subseteq \Z_{N}$ a set and assume $\Phi_{N} \Phi_{N/s} \mid m_S$ and $S \cap ((x+\Z_q) \cup (x+\Z_r))=\{x\}$ for all $x\in S$. Then $S$ is a union of $\Z_p$-cosets.
\end{lem}

\begin{proof}
By Lemma \ref{lemprop1},
the assumption $\Phi_N\Phi_{N/s} \mid m_S$ implies that $S$ satisfies the 3D cube rule on every $\Z_{N/s}$-coset.
By the way of contradiction, assume that $S$ is not the union of $\Z_{p}$-cosets. Hence, there exists an $x \in S$ with $(x+\Z_p)\setminus S \ne \emptyset $. Let $x_p \in (x+\Z_p)\setminus S$. Then using our assumption $S \cap (x+\Z_q)=S \cap (x+\Z_r)=\{x\}$ we obtain that in every 3D cube containing $x$ and $x_p$ in $x+\Z_{N/s}$  there is no element $z\in S$ with $d_H(x,z)=1$. Thus, for any $y \in x+\Z_{N/s}$ with $d_H(x,y)=3$ and $x_p \in C(x,y)$ we have that $y \in S$, by applying the 3D cube rule on $C(x,y)$.

This implies that there are elements $y_1,y_2\in S$ such that $y_1-y_2\in \Z_q$, if $q\ge 3$ (resp., $y_1-y_2\in \Z_r$, if $r\ge 3$), which contradicts the assumption $S \cap ((y_1+\Z_q)\cup (y_1+\Z_r))=\{y_1\}$.
\end{proof}

Note that Lemma \ref{l1} holds verbatim for any permutation of the primes $p,q,r,s$.

\begin{lem}\label{l2} Let $N=pqrs$ for different prime numbers $p,q,r,s$, $(S, \La)$ a spectral pair in $\Z_{N}$ and $r,s\nmid |S|=|\La|$.

\begin{enumerate}
\item\label{iteml2a} If $\Phi_{N} \mid m_{S}$, then  $\Phi_{N/p} \nmid m_{S}$ $($resp. $\Phi_{N/q} \nmid m_{S})$ or $S$ is a tile.
\item\label{iteml2b}
If $\Phi_{N} \mid m_{\La}$, then
$\Phi_{N/p} \nmid m_{\La}$ $($resp. $\Phi_{N/q} \nmid m_{\La})$ or $S$ is tile.
\end{enumerate}

\end{lem}

\noindent
{\it Proof.}
\begin{enumerate}
\item By the way of contradiction, assume $\Phi_{N/p} \mid m_S$. By Lemma \ref{lemprop1}, $\Phi_N \Phi_{N/p}\mid m_S$ implies that the 3D cube rule holds for every $\Z_{N/p}$-coset.
Since $r \nmid |S|$ and $s \nmid |S|$ we have $S\cap((x +\Z_r) \cup (x+\Z_s)) =\{x\}$ for all $x\in S$. Applying Lemma \ref{l1} gives that $S$ is a  union of $\Z_q$-cosets and hence a tile by Summary \ref{summary}. Similar argument applies if $\Phi_{N/q} \mid m_S$.
\item
Assume $\Phi_{N/p} \mid m_{\La}$ (resp. $\Phi_{N/q} \mid m_{\La}$). Similar argument as in the previous case implies that $\La$ is a union of $\Z_q$-cosets (resp. $\Z_p$-cosets). Then, by Summary \ref{summary}, $S$ is a tile.
\qed
\end{enumerate}

\subsection{Reduction \texorpdfstring{$\bmod~ p$}{}}

$\bullet$
The lemma following below is due to Lam and Leung \cite{LL2000}, written in polynomial notation.

\begin{lem}\label{twoprimes}
If $\Phi_{pq}(x)\mid m_S(x)$, then the multiset $S_{pq}$ is the sum of $\Z_{p}$- and $\Z_{q}$-cosets. In polynomial notation,
\[m_S(x)\equiv P(x)\Phi(x^{p})+Q(x)\Phi(x^{q})\pmod{x^{pq}-1},\]
where $P(x),Q(x)\in\Z_{\geq0}[x]$. Then the cardinality of $S$ satisfies $\abs{S}=pk+q\ell$, for some $k,\ell \geq 0$. If $m_S(\xi_{p})m_S(\xi_{q}) \neq 0$, then $k,\ell > 0$, and so
$\abs{S} \geq p+q$.
\end{lem}

$\bullet$ In what follows, it will be very useful to reduce the coefficients of mask polynomials $\bmod \ p$, for some primes $p\mid N$ (i.e., we will consider them as polynomials in $\FF_p[x]$).
For square-free integer $m$ not divisible by $p$ the following equation holds.

\begin{equation}\label{cyclomodp}
\Phi_{pm}(x)=\frac{\Phi_m(x^p)}{\Phi_m(x)}=\frac{\Phi_m(x)^p}{\Phi_m(x)}=\Phi_m(x)^{p-1} \text{ in } \FF_p[x],
\end{equation}
hence
\begin{equation}\label{modpdiv}
\Phi_{pm}(x)\mid m_S(x)\Longrightarrow \Phi_m(x)\mid m_S(x) \text{ in } \FF_p[x].
\end{equation}

\begin{lem}\label{incompletetriangle}
 Assume that
  \[\Phi_{p_2p_3}\Phi_{p_2}\Phi_{p_3}\mid m_S\]
 in $\FF_{p_1}[x]$.
 Then
 at least one of the following conditions holds:
 \begin{enumerate}[{\bf(i)}]
     \item $\Phi_{p_1}\mid m_S$. \label{inciii}
     \item\label{inci} $\Phi_{p_2}\Phi_{p_3}\Phi_{p_2p_3} \mid m_S$.
     \item $\Phi_{p_4}\mid m_{\La}$ for every spectrum $\La$ of $S$. \label{incii}
\end{enumerate}
 If in addition,
 \[m_S(\xi_{p_1})m_S(\xi_{p_4})m_S(\xi_{p_1p_4})\neq0,\]
 then either \ref{inci} holds or \ref{incii} holds along with $p_1\mid\abs{S}$.
\end{lem}

\begin{proof}
It follows from our condition that
\[m_S(x)\equiv\sum_{j=0}^{p_2p_3-1}\abs{S_{j\bmod{p_2p_3}}}x^j\equiv c(1+x+\dotsb+x^{p_2p_3-1})+p_1P(x)~~\bmod(x^{p_2p_3}-1)\]
 for some $0\leq c<p_1$ and $P(x)\in\Z_{\geq 0}[x]$, where $$S_{j\bmod p_2p_3}=\{x\in  S_{p_2p_3}\mid ~x\equiv j \pmod{p_2p_3}\}.$$
Thus
 \[m_S(x)\equiv  c(1+x+\dotsb+x^{p_2p_3-1})~~\bmod(p_1,x^{p_2p_3}-1),\]
 If the polynomial
 $P$ is identically zero, then obviously \ref{inci} holds. If \ref{inci} fails, there is some
 $j$ with $\abs{S_{j\bmod p_2p_3}}\geq c+p_1\geq p_1$. If there is some $j$ such that
 $\abs{S_{j\bmod p_2p_3}}>p_1$, then there are two distinct elements of $S_{j\bmod p_2p_3}$ having
 the same remainder modulo $p_1$, hence $(S-S)\cap\frac{N}{p_4}\znp\neq\vn$ and \ref{incii} holds by
 \eqref{divclass}. If both \ref{inci} and \ref{incii} fail, then we must have
 $\abs{S_{j\bmod p_2p_3}}\leq p_1$ for every $j$, i.e., $c=0$ and
 \[\abs{S_{j\bmod p_2p_3}}=0\text{ or }p_1\]
 for every $j$. Moreover, $(S-S)\cap\frac{N}{p_4}\znp=\vn$ by \eqref{divclass}, so $S_{\frac{N}{p_4}}$ is a proper set, which implies that $\abs{S_{j+kp_2p_3\bmod\frac{N}{p_4}}}\leq1$,
 for every $j$ and $k$. For those $j$ that $\abs{S_{j\bmod p_2p_3}}=p_1$ holds,
 we must have
 \[\abs{S_{j+kp_2p_3\bmod\frac{N}{p_4}}}=1\]
 for all $k$, which in polynomial form is written as
 \[m_S(x)\equiv Q(x)\Phi_{p_1}(x^{\frac{N}{p_1p_4}})\bmod(x^{N/p_4}-1),\]
 which easily yields \ref{inciii}.

 Next, suppose that $m_S(\xi_{p_4})m_S(\xi_{p_1})m_S(\xi_{p_1p_4})\neq0$, so that \ref{inciii} fails.
 By the previous discussion, either \ref{inci} or \ref{incii} holds. Assume that \ref{inci} fails, so
 that \ref{incii} holds true. We only need to show that $p_1\mid\abs{S}$.
 By \eqref{divclass} we obtain
 \[(\La-\La)\cap p_2p_3\Z_N=\set{0},\]
 for every $\La$ spectrum of $S$, so that $\abs{\La}\leq p_2p_3$. On the other hand,
 \[\abs{S}=m_S(1)=cp_2p_3+p_1P(1).\]
 The latter is greater than $p_2p_3$ when $c\geq1$, since $P$ is not identically zero, otherwise \ref{inci} holds.
 Hence $c=0$, and obviously $p_1\mid \abs{S}$, as desired.
\end{proof}

\section{The main result and the structure of the proof}

\begin{Thm}
Fuglede's conjecture is true for $\Z_{pqrs}$, where $p,q,r,s$ are different primes.
\end{Thm}

\textit{Proof.}
Note that the {\it Tile $\Rightarrow$ Spectral} direction follows from  an argument of Meyerowitz and \L aba written on Tao's blog\footnote{\url{https://terrytao.wordpress.com/2011/11/19/some-notes-on-the-coven-meyerowitz-conjecture/\#comment-121464}} and also by Ruxi Shi \cite{Shi18}.

\medskip

Now we consider the {\it Spectral $\Rightarrow$ Tile} direction. Let $(S,\La)$ be a spectral pair of $\Z_{pqrs}$. Let $m_S$ and $m_{\La}$ denote their mask polynomials. In the proof we distinguish the following cases.
\begin{enumerate}[(I)]
    \item
    $\Phi_{N}\mid m_S$ or $\Phi_{N}\mid m_{\La}$,
    \item
    $\Phi_{N}\nmid m_S$ and $\Phi_{N}\nmid m_{\La}$.
    \qedhere
\end{enumerate}

\subsection{Proof of Case (I):  \texorpdfstring{$\Phi_N \mid m_S$}{} or  \texorpdfstring{$\Phi_N \mid m_{\La}$}{}}

We also assume that $r,s\nmid |S|=|\La|$ hence, by Corollary \ref{rsndiv}, $\Phi_r, \Phi_s\nmid m_S, m_{\La}$.

\begin{lem}\label{lem:cosets} Let $(S, \La)$ be a spectral pair of $\Z_N$. Assume that for every $x \in S$ we have $x +\Z_p \subseteq S$ or $x+\Z_q \subseteq S$. Then $S$ is a tile.
\end{lem}

\begin{proof} By Summary \ref{summary}, if $S$ is the union of $\Z_p$-cosets (or $\Z_q$-cosets), then $S$ is a tile.
Since $S$ is primitive, $S$ contains $x+\Z_p$ and $y+\Z_q$ with $(x+\Z_p) \cap (y+\Z_q)= \emptyset$.

First, assume that $(x+\Z_p) \cup (y+\Z_q)$ is contained in a $\Z_{N/s}$-coset. Then it is easy to find $x_1 \in (x+\Z_p), ~ x_2 \in (y+\Z_q)$ with $x_1-x_2 \in \Z_r$, implying $r \mid |\La|$, a contradiction. Note that a similar argument works if $(x+\Z_p) \cup (y+\Z_q)$ is contained in a $\Z_{N/r}$-coset.

Now assume that $(x+\Z_p)\cup (y+\Z_q)$ is not contained in any proper coset of $\Z_N$. Then we may write $(x+\Z_p)=\left\{ (a,b_1,c_1,d_1) \mid a \in \Z_p \right\}$ and $(y+\Z_q)=\left\{ (a_2,b,c_2,d_2) \mid b \in \Z_q \right\}$, where $c_1 \ne c_2$ and $d_1 \ne d_2$. Then by suitable choice of $a$ and $b$ we obtain that $\Phi_N \mid m_{\La}$ and $\Phi_{\frac{N}{p}} \mid m_{\La}$.
Using Lemma \ref{l2}.\ref{iteml2b} we obtain that $S$ is a tile.
\end{proof}

\begin{lem}\label{lem:phimidS}
Let $(S, \La)$ be a spectral pair of $\Z_N$.
If $\Phi_N\mid m_S$ and there exists an $x \in S$ such that $x +\Z_p \not\subseteq S$ and $x+\Z_q \not\subseteq S$, then $S$ is tile.
\end{lem}
\begin{proof}
Let us assume that there are $x,y,z \in \Z_N$ such that $x \in S$, $y,z \not\in S$ and $y \in x+\Z_p$, $z \in x+ \Z_q$. As $\Phi_N\mid m_S$, the 4D cube rule holds for $S$. Then, by taking any 4D cube $C$ having vertices $x,y$ and $z$, we obtain that $S$ contains at least one element of $C$ of Hamming distance $3$ from $x$. Indeed, points of Hamming distance $1$ are excluded by $y,z \not\in S$ and $r \nmid |S|$, $s \nmid |S|$. Without loss of generality we may assume $r>s$.

Fix a 4D cube $C$ containing $x,y,z$.
Let $u=u(x,y,z)$ be the point of $C$ such that $d_H(u,x)=3$ and $u$ and $x$ have the same $r$-coordinate. Let $o$ be the point opposite to $x$ on $C$.

{\bf Case 1.} $s\ge 3$ (hence $r\ge 5$).

\smallskip
{\bf Case 1.1.} If $u \not\in S$, then one of the other $3$ points of distance $3$ from $x$ is in $S$. (These are the ones that differ from $x$ in their $r$-coordinate.) Since $r\ge 5$, modifying only the $r$-coordinate of $o$ we may build up $3$ more 4D cubes having vertices $x,y,z,u$.
Applying the previous argument using $y,z,u \not\in S$ we obtain that $S$ contains a pair of points of Hamming distance $3$ from $x$ differing only in their $r$-coordinate, which contradicts the fact that $r\nmid |S|$.

{\bf Case 1.2.} If $u \in S$, then let $u'$ be a point which only differs in its $s$-coordinate from $u$ and have different $s$-coordinate from $x$.
Such $u'$ exists since $s \ge 3$.
Then $u' \not\in S$ since $s \nmid |S|$. Hence we may apply the previous argument by exchanging the role of $u$ and $u'$.

\medskip

Note that similar argument works if $p\nmid |S|$ (resp. $q\nmid |S|$) by changing the role of $p$ (resp.~$q$) and $s$. Therefore we can assume that $pq\mid |S|$.

\medskip

{\bf Case 2.} $s=2$ (and $pq \mid |S|$).

Note that the projection $S_{pqs}$ is a set in $\Z_{pqs}$ as $r \nmid |S|$. Therefore $|S|$ can be either $pq$ or $spq=2pq$. The latter case is not possible since $s\nmid |S|$, thus $|S|=|\La|=pq$. Now we project $\La$ on $\Z_{2pq}$ and so at least half of the elements of $\La$ project on the same $\Z_{pq}$-coset. Easy calculation shows if $\min\{p,q\} \ge 3$ (which clearly holds), then there are three pairs of elements of $\La_{2pq}$ having only different $p$-, $q$-, $pq$-coordinates, respectively.

Since $(S,\La)$ is a spectral pair, these conditions imply that $\Phi_p\Phi_q\Phi_{pq}\mid m_{S}$ over $\mathbb{F}_r[x]$.
Applying Lemma~\ref{incompletetriangle} and $rs\nmid|S|$  which gives $\Phi_r, \Phi_s\nmid m_S, m_{\La}$, we conclude that ${\bf(ii)}$ holds, i.e., $\Phi_p\Phi_q\Phi_{pq}\mid m_{S}$. This means that $S_{pq}=c\Z_{pq}$, for some $c\in \mathbb{N}$.  As $|S|=pq$, we have that $c=1$ and plainly, $S$ tiles $\Z_N$.  \end{proof}

\begin{lem}\label{lem:phinmidS}
Let $N=pqrs$ and
let $(\La, S)$ be a spectral pair.
Assume that $r,s\nmid|S|=|\La|$,  $\phi_N\nmid m_S$ and $\phi_N\mid m_{\La}$.
Then $S$ is a tile.
\end{lem}

\begin{proof}
By applying Lemmata \ref{lem:cosets} and \ref{lem:phimidS} to the spectral pair $(\La, S)$, we obtain that $\La$ is a tile.
Thus $|\La|\in \{1,p,q,pq\}$ and by $\cite{Shi18}$ we have $\La_{|\La|}=\Z_{|\La|}$. Note that one element subsets are tiles, thus we assume $|\La|\ne 1$.

We distinguish two cases according to the conditions of the previous two Lemmata.

\medskip
{\bf Case 1.} Assume that for every $x \in \La$ we have $x +\Z_p \subseteq \La$ or $x+\Z_q \subseteq \La$.

If $q \nmid |\La|$ (i.e. $|\La|=p$), then there is no point $x$ in $\La$
such that $x+\Z_q \subset \La$. Thus $\La$ is a $\Z_p$-coset, hence $\La-\La\in \Z_p$. Therefore, by spectrality, $S$ intersects every $\Z_{qrs}$-coset once, when $S$ is a tile. A similar argument works if $|\La|=q$.

It remains to handle the case $|S|=|\La|=pq$.
Assume first that $(x+\Z_p) \cup (x+\Z_q)\subset \La$. Then $\Phi_p\Phi_q\Phi_{pq} \mid m_S$. Thus $S_{pq}=\Z_{pq}$ and since $|S|=pq$ we obtain that $S$ is a tile.

Assume now that there are $x$ and $y$ such that $x+\Z_p$ and $y+\Z_q$ are disjoint subsets of $\La$. Then $\La_{pq}\ne \Z_{pq}$, a contraction.

Finally, if $\La$ is the union of $\Z_p$-cosets or $\Z_q$-cosets only, then $S$ is a tile by Summary \ref{summary}.

\medskip
{\bf Case 2.} There exists an $x \in \La$ such that $x +\Z_p \not\subseteq \La$ and $x+\Z_q \not\subseteq \La$

Note that the proof of Lemma \ref{lem:phimidS} shows that there is no such spectral set except if $|\La|=pq$ and $s=2$.
Since $\La$ is tile we have $\La_{pq}=\Z_{pq}$. Since $p\ge 3$, there are at least $3$ elements of $\La_{pq}$ in each $\Z_p$-cosets. Two of them have the same $s$-coordinate as well since $s=2$. Thus $\Phi_p \mid m_S$ or $\Phi_{pr}\mid m_S$. Similarly we obtain  $\Phi_q \mid m_S$ or $\Phi_{qr}\mid m_S$. There are $3$ elements of $\La_{pq}$ such that the Hamming distance of any two of them is $2$. Thus again two of them have the same $s$-coordinate. This implies that $\Phi_{pq} \mid m_S$ or $\Phi_{pqr}\mid m_S$. Thus  $\Phi_p\Phi_q\Phi_{pq} \mid m_S$ in $\mathbb{F}_r[x]$.

By applying Lemma \ref{incompletetriangle} we obtain
$\Phi_p\Phi_q\Phi_{pq} \mid m_S$ since $r \nmid |S|$ and $s \nmid |S|$ (so cases \ref{inciii} and \ref{incii} are excluded).
Thus $S_{pq}=\Z_{pq}$ and since $|S|=pq$, we obtain that $S$ is a tile.

\end{proof}

\subsection{Proof of Case (II):   \texorpdfstring{$\Phi_{pqrs}\nmid m_S$}{} and \texorpdfstring{$\Phi_{pqrs}\nmid m_{\La}$}{}}

We start with a technical lemma as follows.
\begin{lem}\label{primitivedifference}
Let $N=p'p''q'q''$ and let $B\ssq\Z_N$ be primitive. Then, for every pair $p',p''$ of distinct primes with $p'p''\mid N$ we have
\[(B-B)\cap\bra{\Z_N^{\star}\cup p'\Z_N^{\star}\cup p''\Z_N^{\star}\cup p'p''\Z_N^{\star}}\neq\vn.\]
\end{lem}

\begin{proof}
Let $b\in B$; since $B$ is not primitive, $b-B$ is not contained in either $q'\Z_N$ or $q''\Z_N$, where $P=\set{p',p'',q',q''}$. So, let $b-b'\notin q'\Z_N$ and
$b-b''\notin q''\Z_N$. If either $b-b'\notin q''\Z_N$ or $b-b''\notin q'\Z_N$ then either $b-b'$ or $b-b''$ belongs to
\[\Z_N^{\star}\cup p'\Z_N^{\star}\cup p''\Z_N^{\star}\cup p'p''\Z_N^{\star}\]
as desired. On the other hand, if $b-b'\in q''\Z_N$ and $b-b''\in q'\Z_N$, then $b'-b''$ belongs to the aforementioned union, completing the proof.
\end{proof}
\begin{cor}\label{cor35}
Let $N=pqrs$ and let $(S,\La)$ be a spectral pair in $\Z_N$ such that $S$ is primitive. Then there exist $x$ and $y$ in $S$ such that they have different $p$- and $q$-coordinates.  With other words $\Phi_{N} \mid m_{\La}$ or $\Phi_{\frac{N}{r}} \mid m_{\La}$ or $\Phi_{\frac{N}{s}} \mid m_{\La}$ or $\Phi_{\frac{N}{rs}} \mid m_{\La}$.
\end{cor}
Note that if $S$ is a spectral set, then the previous corollary can be applied to both $S$ and $\La$ since $\La$ is also spectral in this case.

Now, we handle a special case which will be crucial in the sequel.

\begin{prop}\label{phipqdiv}
Let $S \subset \Z_{pqrs}$ be a spectral set that is primitive and suppose $\Phi_p\Phi_q \mid m_S$ and $\Phi_N \nmid m_S$. Then $S$ is a tile.
\end{prop}

\begin{proof}
We prove this statement in two steps.
\begin{lem}\label{lpq1}
Under the assumptions of Proposition \ref{phipqdiv} we have $S_{pq}=c\Z_{pq}$ for some $c\in \mathbb{N}$.
\end{lem}
\begin{proof}
By Corollary \ref{cor35}, if $\Phi_N\nmid m_S$, then either $\Phi_{pq}\mid m_S$, or $\Phi_{pqr}\mid m_S$ or $\Phi_{pqs}\mid m_S$.
If $\Phi_p \Phi_q \Phi_{pq}\mid m_S$, then we are done. Thus, by contradiction, let us assume $\Phi_{pq} \nmid m_S$. By the symmetry of the role of $r$ and $s$ we can assume that $\Phi_{pqr} \mid m_S$.
Thus \[\Phi_p(x)\Phi_q(x) \Phi_{N/s}(x) = \Phi_p(x) \Phi_q(x) \Phi_{pq}(x)^{r-1} \mid m_S(x) \textrm{ in }\Z_r[x].\]
We may apply Lemma \ref{incompletetriangle}.
Since $r \nmid |S|=|\La|$ and $s \nmid |S|=|\La|$ we obtain $\Phi_p \Phi_q\Phi_{pq} \mid m_S$.
\end{proof}

\begin{lem}\label{lpq2}
Under the assumptions of Proposition \ref{phipqdiv}
if $S_{pq}=c\Z_{pq}$ for some $0<c\in \N$, then $c=1$.
\end{lem}

\begin{proof}
We proceed by contradiction. Take two elements $t_1, t_2 \in S$ with the same $p$- and $q$-coordinates. Their $r$- or $s$-coordinates must be different since otherwise we have  $\Phi_{s} \mid m_{\La}$ or $\Phi_r \mid m_{\La}$, which was excluded. So let $t_1=(a_1,a_2,a_3,a_4)$ and $t_2=(a_1,a_2,a_3',a_4') \in S$, where $a_3 \ne  a_3'$ and $a_4 \ne  a_4'$.

As $p,q\ge 2$ and $c>1$, clearly there are elements $t_3, t_4$ having $a'_1$ in the first and $a'_2$ in their second coordinate, where $a_1 \ne a'_1$ and $a_2\ne a'_2$.
Since $\Phi_N\nmid m_{\La}$, it implies that there are no points of Hamming-distance 4 in $S$. Thus,
we may assume that $t_3=(a'_1,a'_2,a_3,a'_4),  t_4=(a'_1,a'_2,a'_3,a_4)$, since $t_3$ and $t_4$ cannot differ only in their $r$- or $s$-coordinate.

 If $p, q\ge 3$, then there are $a''_1 \not\in \{a_1, a'_1\}$ and $a''_2\not\in \{a_2, a'_2\}$.
 Then in the $\Z_{rs}$-coset $(a_1'',a_2'',*,*)$ every element is of Hamming distance $4$ from one of $t_1,t_2,t_3,t_4$, a contradiction.

If $c\ge 3$, then we also get a contradiction, as there is an element $t=(a_1, a_2, a''_3, a''_4)$ with  $a''_3 \not\in \{a_3, a'_3\}$ and $a''_4\not\in \{a_4, a'_4\}$ and so $d_H(t, t_4)=4$.

Hence we can assume without loss of generality that $c=2$ and $p=2$. So every element has either $a_1$ or $a'_1$ in its first coordinate. Iterating the argument as above, we get that for every element having $a_1$ in its first coordinate is of the form $(a_1, *, a_3, a_4)$ or $(a_1,*, a'_3, a'_4)$. Similarly, every element having $a'_1$ in its first coordinate is of the form $(a'_1, *, a_3, a'_4)$ or $(a'_1,*, a'_3, a_4)$. Moreover all of these elements are in $S$. This implies that $|S|=4q$. Thus $S$ is not a tile. Therefore we have to show that $S$ is not a spectral set.

By contradiction, assume that $(S, \La)$ is a spectral pair for some $|S|=|\La|=4q$. Then there is a $\Z_{pq}$-coset which contains at least 2 elements of $\La$, i.e., these elements have the same $p$- and $q$-coordinates. Then their $r$- and $s$- coordinate should be different, otherwise $\Phi_r \mid m_S$ or $ \Phi_s \mid m_{S}$, which is excluded.  This implies that $\Phi_{rs}\mid m_{S}$. Hence the 2D cube rule holds for $S_{rs}$ on $\Z_{rs}$, i.e., $S_{rs}$ is the sum of $\Z_r$- and $\Z_s$-cosets. All the elements of $S$ can have only 2 different $r$-coordinates and 2 different $s$-coordinates, thus $S_{rs}$ does not contain either a $\Z_r$-coset or a $\Z_s$-coset. This contradiction shows that $S$ is not a spectral set.
\end{proof}

\noindent
By Lemma \ref{lpq1}, $S_{pq}=c\Z_{pq}$ for some $0<c\in \mathbb{N}$. By Lemma \ref{lpq2}, $c=1$, and then $S$ is a tile, finishing the proof of Proposition \ref{phipqdiv}.
\end{proof}

\begin{lem}\label{pqrsla} Let $(S,\La)$ be a spectral pair of $\Z_{pqrs}$. Assume $\Phi_N \nmid m_S$ and $\Phi_N \nmid m_{\La}$. Then
$$\Phi_{pqr}\mid m_S \Longleftrightarrow \Phi_{pqr}\mid m_{\La}.$$
$$\Phi_{pqs}\mid m_S \Longleftrightarrow \Phi_{pqs}\mid m_{\La}.$$
\end{lem}
\begin{proof}
The proof is the same for the two statements and by duality it is enough to show that if $\Phi_{pqr}\mid m_S$, then there is a pair of points in $S_{pqr}$ of Hamming distance 3. This implies that the preimage of these points are of distance 3 in $S$ as $\Phi_N\nmid m_S$, which implies $\Phi_{pqr}\mid m_{\La}$.

We prove it by contradiction. If there is no pair of points of Hamming distance 3 in $S_{pqr}$, then either all of the points are of distance 1 from each other, but then one can easily see that $S$ is not primitive or there is a pair of points $x,y\in S_{pqr}$ of Hamming distance 2.
Since the 3D cube rule  holds on $S_{pqr}$ and $S_{pqr}$ is a set (as $\Phi_s \nmid m_S$), there are at least two other points in any 3D cube containing both $x$ and $y$ of odd distance from $x$. There are $4$ points of odd Hamming distance from $x$ on each of these 3D cubes. Two of them are excluded, since each of them are of Hamming distance $3$ from either $x$ or $y$.
The remaining two points that are on the 2D cube spanned by $x$ and $y$ must be in the set $S_{pqr}$ (by 3D cube rule).
Then all of the points of $\Z_{pqr}$ which is not in the plane determined by $x,y$ is of Hamming distance 3 from one of these 4 points, so none of them can be in $S_{pqr}$.
By Summary \ref{summary}, $S$ is primitive, thus $S_{pqr}$ should also be primitive, a contradiction.
\end{proof}

\begin{lem}\label{lemris}
Let $(S,\La)$ be a spectral pair of $\Z_N$. Assume $\Phi_N \nmid m_S$ and $\Phi_N \nmid m_{\La}$. If $\Phi_{pqr}\mid m_{\La}$ $($resp. $\Phi_{pqs}\mid m_{\La})$, then every pair of points in $\La$ having different $p$- and $q$-coordinates must have different $r$-coordinate $($ resp. $s$-coordinate$)$ as well or $S$ is a tile.
\end{lem}

\begin{proof}
Suppose there are points $x,y\in \La$ with different  $p$- and $q$-coordinates and with the same $r$-coordinate. Then their projection $x',y'\in\La_{pqr}$ is of Hamming distance 2. By assumption the 3D cube rule holds for $\La_{pqr}$, thus there are at least two more points $z',w'$ in every 3D cube containing $x'$ and $y'$.

We have two cases. If one of these additional points are in 2D cube determined by $x',y'$,
then this implies that $\Phi_p \Phi_q \Phi_{pq}\mid m_S$ over $\mathbb{F}_s$.
Thus by Lemma \ref{incompletetriangle} we have that either $\Phi_p\Phi_q\mid m_S$ and $\Phi_{pq} \mid m_S$, then $S$ is a tile and we are done by Proposition \ref{phipqdiv}; or $\Phi_r\mid m_S$ or $\Phi_s\mid m_{\La}$, which cases were excluded.
Otherwise both of the additional points $z',w'$ differ in their $r$-coordinates from $x', y'$, respectively. Without loss of generality we can assume that $d_H(x', w')=d_H(y',z')=3$.
Hence, the preimage $z$ of $z'$ has the same $s$-coordinate as $y$; and the preimage $w$ of $w'$ have the $s$-coordinate as $x$, because their distance in $\Z_{pqrs}$ is at least $3$ and it cannot be 4, since $\Phi_s\nmid m_S$. Thus $d_H(x, w)=d_H(y,z)=3$.

If $r\ge 3$, then taking another cube containing $x'$ and $y'$ we get $z''$ and $w''$ similarly as above and then $z'$ and $z''$ differ only in their $r$-coordinates, thus $\Phi_r\mid m_S$, which was excluded and we are done.

If $r=2$, then we have two cases whether $x$ and $y$ have the same or different $s$-coordinates. If they have the same $s$-coordinates, then so do $z$ and $w$. By the first part of the proof, there is no more points of the 3D cube $C(x',w')$ in $\La_{pqr}$. On the other hand, using the fact that $r=2$, one can see that every point of $\Z_{pqr}\setminus C(x',w')$ is of Hamming distance 3 from at least one of $x',y',z',w'$. Hence, every point of $S$ has the same $s$-coordinate thus $\La$ is not primitive so $S$ is a tile, by Summary \ref{summary}.
Assume now that $x'$ and $y'$ have different $s$-coordinates. This implies that $\Phi_{pqs}\mid m_S$. By Lemma \ref{pqrsla}, it implies that $\Phi_{pqs}\mid m_{\La}$ and the projection of $x$ and $w$ in $\La_{pqs}$ is of distance 2. Repeating the argument above for these case and using $s>2$ we get a contradiction.
\end{proof}

\begin{prop}\label{proplatilestile}
Let $(S,\La)$ be a spectral pair. Assume that $\Phi_N \nmid m_S$, $\Phi_N \nmid m_{\La}$ and $\La$ is a tile. Then $S$ is a tile.
\end{prop}
\begin{proof}
We may assume $|S|=|\La|=1 \mbox{ or } p \mbox{ or } pq$. The $|S|=1$ case is trivial.

\smallskip

If $|\La|=pq$, then $\La_{pq}=\Z_{pq}$ by \cite{Shi18}. Since $\Phi_N \nmid m_S$ we have that for every pair of elements $x,y$ of $\La$ if the $p$-coordinate and $q$-coordinate of $x$ and $y$ are different, then their $r$- or $s$-coordinate is the same.

If $\La_{pq}=\Z_{pq}$, then there is a pair of points having different $p$- and $q$- that have different $r$-coordinate as well. Otherwise, $\La$ is not primitive, and then $S$ is a tile by Summary \ref{summary}. Let $x_1$ and $y_1$ be such a pair in $\La$, so their $p$-, $q$- and $r$-coordinates are different. So the $s$-coordinate of $x_1$ and $y_1$ is the same since $\Phi_N \nmid m_S$. Thus $\Phi_{pqr} \mid m_S$, which implies $\Phi_{pqr} \mid m_{\La}$ by Lemma \ref{pqrsla}. Thus by Lemma \ref{lemris} if the $p$-
and $q$-coordinates of a pair of points of $\La$ are different, then so do their $r$-coordinates. Again, we obtain that the $s$-coordinate of these points is the same. Then it is easy to see from $\La_{pq}=\Z_{pq}$ that the $s$-coordinate of every point of $\La$ is the same. Thus $\La$ is not primitive, and then $S$ is a tile.

\smallskip

Finally, assume $|\La|=p$. Suppose by contradiction that $S$ is not a tile. Then there is a pair $x$ and $y$ in $S$ such that $p \mid x-y$. Then by our assumptions we have $\Phi_{qr} \mid m_{\La}$ or $\Phi_{qs} \mid m_{\La}$ or $\Phi_{rs} \mid m_{\La}$ or $\Phi_{qrs} \mid m_{\La}$. We may assume without loss of generality that $\Phi_{rs} \mid m_{\La}$ or $\Phi_{qrs} \mid m_{\La}$.

If $\Phi_{rs} \mid m_{\La}$, then $\La_{rs} $ is the sum of $\Z_r$-cosets and $\Z_s$-cosets by Lemma \ref{twoprimes}.
Since $r \nmid |\La|$ and $s \nmid |\La|$ both types appear in the sum. Thus we have $\Phi_{pq} \mid m_S$ or $\Phi_{q} \mid m_S$ or $\Phi_{p} \mid m_S$.

The last option would imply that $S$ is equidistributed $\bmod{p}$ and since $|S|=p$, it is a tile, which contradicts our assumption.
$\Phi_q \mid m_S$ is excluded since $q \nmid |S|$. Finally, one can see that $\Phi_{pq} \mid m_S$ and $|S|=p$ also implies that $S_{pq}$ is a $\Z_p$-coset by Lemma \ref{twoprimes}, and hence $S$ is a tile, a contradiction.

Assume $\Phi_{qrs} \mid m_{\La}$. Then the 3D cube rule holds for $\La_{qrs}$.
We claim that in this case there is a pair of points in $\La_{qrs}$ of Hamming distance $1$. This follows simply from the fact that if $z \in \La_{qrs}$ such that none of the point of Hamming distance $1$ from $z$ is in $\La_{qrs}$, then every point $z'$ with $d_H(z,z')=3$ is in $\La_{qrs}$ by the 3D cube rule.

We obtain that $\Phi_{up} \mid m_S$ or $\Phi_u \mid m_S$, where $u \in \{ q,r,s\}$. The latter is excluded by $u  \nmid |S|$. If $\Phi_{up} \mid m_S$ and $|S|=p$, then by Lemma \ref{twoprimes}, $S_{up}$ is a $\Z_p$-coset so $S$ is a tile, which is a contradiction.
\end{proof}

\begin{cor}\label{corris}
Let $(S,\La)$ be a spectral pair of $\Z_N$. Assume $\Phi_N \nmid m_S$ and $\Phi_N \nmid m_{\La}$. If $\Phi_{pqr}\mid m_S$ $($ resp. $\Phi_{pqs}\mid m_S)$, then every pair of points in $\La$ having different $p$- and $q$-coordinates must have different $r$-coordinate $($ resp. $s$-coordinate$)$ as well or $S$ is a tile.
\end{cor}

\begin{proof}
Applying Lemma \ref{pqrsla}, either the statement follows directly from  Lemma \ref{lemris} or $\La$ is a tile. Then by Proposition \ref{proplatilestile}, $S$ is a tile.
\end{proof}

\begin{lem}\label{lems2}
If $\Phi_{pqr}\mid m_S$, then $\Phi_{pqs}\nmid m_S$, and also if $\Phi_{pqs}\mid m_S$, then $\Phi_{pqr}\nmid m_S$. The same holds for $\La$.
\end{lem}
\begin{proof}
We take $x$ and $y$ obtained from Corollary \ref{cor35}, i.e., they have different $p$- and $q$-coordinates. By Corollary \ref{corris}, if $\Phi_{pqr}\mid m_S$ holds, then their $r$-coordinates are different, as well. Similarly, if $\Phi_{pqs}\mid m_S$ holds, then so their $s$-coordinates. If both $\Phi_{pqr}\Phi_{pqs}\mid m_S$ holds, then $d_H(x,y)=4$, which is a contradiction. The statement for $\La$ directly follows by applying Lemma \ref{pqrsla}.
\end{proof}

By Lemma \ref{lems2}, without loss of generality, we can assume that $\Phi_{pqs}\nmid m_S$, and thus $\Phi_{pqs}\nmid m_{\La}$ by Lemma \ref{pqrsla}.
\begin{lem}\label{lems3}
If $\Phi_{pqs}\nmid m_S$, then $S$ is not primitive.
\end{lem}
\begin{proof}
Let $x,y\in S_{pqs}$ be the projection of the points getting from Corollary \ref{cor35}, i.e., their $p$- and $q$-coordinates are different. Then their $s$-coordinate is the same, otherwise $\Phi_{pqs}\mid m_{\La}$. Now we take those points that have different $s$-coordinate than $x$ (and $y$). These points must  have the same $p$-coordinate as $x$ has and the same $q$-coordinate as $y$ has or vica-versa. Thus, there are two $\Z_s$-cosets in $S_{pqs}$ that contain all of the points having different $s$-coordinate than $x$ has.
We can distinguish three cases:

{\bf Case 1.} Every point have the same $s$-coordinates as $x$ has.
Then, clearly, $S$ is not primitive.

{\bf Case 2.} There are points on both $\Z_s$-cosets.
Then taking one point from each coset, they differ in their $p$- and $q$- coordinates and since $\Phi_{pqs}\nmid m_{\La}$ and $\Phi_{N}\nmid m_{\La}$, their $s$-coordinates must be the same. Let  $z$ and $w$ denote this pair of points. By the same reasoning it follows that every point having different $s$-coordinate than $z$ (and $w$) are in two $\Z_s$-cosets containing $x$ and $y$, respectively.

Observe that if $a \in S$ and $b \in S$  have different $p$- and $q$-coordinates, then the $\Z_s$-cosets containing them do not contain any other element of $S$.

Therefore it immediately follows that $S_{pqs}=\{x,y,z,w\}$, thus $|S|=4$ and by Summary~\ref{summary}, $|S|$ is not spectral.

{\bf Case 3.} One $\Z_s$-coset contains all points of $S$ having different $s$-coordinate than $x$.

Now our strategy is that under the conditions $\phi_N\nmid m_S$ and $\phi_{pqs} \nmid m_S$ and the assumption of Case 3 we prove that $\phi_{pqr} \nmid m_S$ also holds. Thus,  it is enough to exclude that $\phi_{pq} \mid m_S$, by Corollary \ref{cor35}.

We denote one of the points of $S$ from this $\Z_s$-coset by $z$. Without loss of generality we can assume that $z$ has the same $q$-coordinate as $x$ and $p$-coordinate as $y$.

It is easy to see that if $\Phi_{pqs}\nmid m_{\La}$ and $x,y,z \in S_{pqs}$, then every point of $S_{pqs}$ are in $(x+\Z_p) \cup ( y+\Z_q )\cup (z+\Z_s)$.
Hence every point having the same $q$-coordinate as $x$ and different from $y$ has the same $s$-coordinate as $x$, and the every point having the same $p$-coordinate as $y$ and different from $x$ has the same $s$-coordinate as $y$.

Since $S_{pqs}$ is a set, every element in $S_{pqs}\cap (x+\Z_p)$ has multiplicity 1. As $y$ and $z$ have the same $p$-coordinates, it implies that $S$ is not equidistributed by $p$, and hence $\Phi_p \nmid m_S$. Similar argument shows that $\Phi_q\nmid m_S$.
Furthermore, $S_{pq}$ is in the projection of $(x+\Z_p) \cup ( y+\Z_q )$, where their intersection may have higher multiplicity.

As $\Phi_{pqs}\nmid m_{\La}$ holds, by Lemma \ref{pqrsla}, the same argument shows that $\La_{pqs}\subset (u+\Z_p) \cup ( v+\Z_q )\cup (w+\Z_s)$ for some $u,v,w\in\La_{pqs}$ (otherwise $\La$ is primitive and by Summary \ref{summary} we are done). In particular, $\Phi_p\nmid m_{\La}$ and $\Phi_q\nmid m_{\La}$ hold, as well.

Now we show that $\Phi_{pqr}\nmid m_S$.
By contradiction we assume that $\Phi_{pqr}\mid m_S$.
In this case the 3D cube rule holds. Now we take the preimages of $x,y$ and project it to $S_{pqr}$ which is also a set by $s \nmid |S|$.
Denote these projections by $x',y'\in S_{pqr}$. As their $p$-, $q$-coordinates are different, by Corollary \ref{corris}, their $r$-coordinates are also different. Let us write $x'=(a_1,b_1,c_1)$ and $y'=(a_2,b_2,c_2)$, where $a_i \in \Z_p$, $b_i \in \Z_q$ and $c_i \in \Z_r$ are pairwise different for $i=1,2$. It is clear that $p\ge 3$ or $q\ge 3$.

Without loss of generality, we can assume that $p\ge 3$. Thus there exists an $a_3\in \Z_p\setminus \{a_1,a_2\}$. Clearly, $z_1=(a_3, b_2,c_1)$ and $z_2=(a_3,b_1, c_2)$ are not in $S_{pqr}$ since $x'$ and $z_1$ (resp., $y'$ and $z_2$) differ in their $p$- and $q$-coordinates and coincide in their $r$-coordinate, which is impossible by Corollary \ref{corris}.
On the other hand, $z_3=(a_3, b_1,c_1)$ cannot be in $S_{pqr}$, since $z_3$ has the same $q$-coordinate as $x'$ and different from $y'$, which implies that the preimage of $z_3$ in $\Z_N$ and $x'$ have the same $s$-coordinate.
Thus $x'$ and $z_3$ are in the same $\Z_p$-coset, which implies that $\Phi_p\mid m_{\La}$ which was excluded. Similar argument shows that $z_4=(a_3,b_2,c_2)\not \in S_{pqr}$ since $y' \in S_{pqr}$, by changing the role of $x'$ and $y'$.
\setlength{\unitlength}{0.9cm}
\thicklines
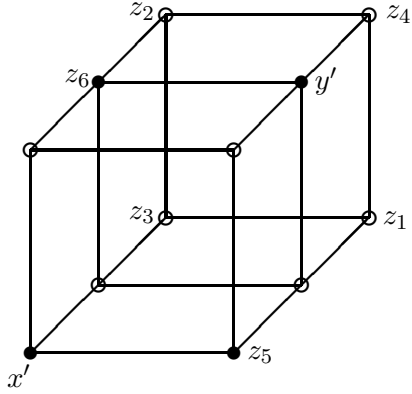
\begin{center}
\begin{picture}(8,8)
\put(2,2.2){\line(1,0){3}}
\put(1.65,1.7){$x'$}
\put(2,2.2){\line(0,1){3}}
\put(2,2.2){\circle*{0.2}}
\put(5,2.2){\circle*{0.2}}
\put(5.2,2.11){$z_5$}
\put(5,2.2){\line(0,1){3}}
\put(2,5.2){\circle{0.2}}
\put(5,5.2){\circle{0.2}}
\put(2,5.2){\line(1,0){3}}

\put(3,3.2){\line(1,0){3}}
\put(3,3.2){\line(0,1){3}}
\put(3,3.2){\circle{0.2}}
\put(6,3.2){\circle{0.2}}
\put(6,3.2){\line(0,1){3}}
\put(3,6.2){\circle*{0.2}}
\put(2.50,6.2){$z_6$}
\put(6,6.2){\circle*{0.2}}
\put(6.2,6.1){$y'$}
\put(3,6.2){\line(1,0){3}}

\put(4,4.2){\line(1,0){3}}
\put(3.45,4.18){$z_3$}
\put(4,4.2){\line(0,1){3}}
\put(4,4.2){\circle{0.2}}
\put(7,4.2){\circle{0.2}}
\put(7,4.2){\line(0,1){3}}
\put(7.2,4.1){$z_1$}
\put(4,7.2){\circle{0.2}}
\put(7,7.2){\circle{0.2}}
\put(7.25,7.15){$z_4$}
\put(4,7.2){\line(1,0){3}}
\put(3.45,7.19){$z_2$}

\put(2,2.2){\line(1,1){2}}
\put(5,2.2){\line(1,1){2}}
\put(2,5.2){\line(1,1){2}}
\put(5,5.2){\line(1,1){2}}
\end{picture}
\end{center}
As $(a_2,b_2,c_1)$ and $(a_1,b_1,c_2)$ are not in $S_{pqr}$ by Corollary \ref{corris}, applying the 3D cube rule for $C(x',z_4)$ and $C(y',z_3)$ we get that $z_5=(a_1, b_2,c_1)$ and $z_6=(a_2,b_1,c_2)$ are in $S_{pqr}$. Then the projection of these 4 points $\{x',y',z_5,z_6\}$ in $S_{pq}$ is a 2D cube, although it is known that $S_{pq}$ is in  $(x+\Z_p) \cup ( y+\Z_q )$. This contradiction shows that $\Phi_{pqr}\nmid m_S$.

By Corollary \ref{cor35}, if $\Phi_N\nmid m_S$ and $\Phi_{pqr}\nmid m_S$ and  $\Phi_{pqs}\nmid m_S$, then $\Phi_{pq}\mid m_S$.
In this case, the projection $S_{pq}$ is the sum of $\Z_p$- and $\Z_q$-cosets. As $S_{pqs}\subset(x+\Z_p) \cup ( y+\Z_q )\cup (z+\Z_s)$, it follows that $S_{pq}$ is the sum of a $\Z_p$- and a $\Z_q$-coset.
Let $x'', y''$, and $z''$ denote the projection of $x, y$ and $z$ on $\Z_{pq}$, respectively. Clearly, $z''=(x''+\Z_p)\cap (y''+\Z_q)$.
Since all elements of $ (x''+\Z_p)\setminus \{z''\}$ have different $p$- and $q$-coordinates from any elements of $(y''+\Z_q)\setminus \{z''\}$ and $\Phi_{pqr}, \Phi_{pqs}\nmid m_S$, the preimages of the elements $(x''+\Z_p)\setminus \{z''\}$ has the same $r$- and $s$-coordinate. Since the elements of $(x''+\Z_p)$ have also the same $q$-coordinate, it follows that their preimages differ only in their $p$-coordinates. As $p\ge 3$ (by our assumption), we have $|(x''+\Z_p)\setminus \{z''\}|\ge 2$, and hence $\Phi_p\mid m_{\La}$, which was excluded.
This contradiction shows the statement.
\end{proof}

\section*{Acknowledgement}
G. Kiss was supported by a Premium Postdoctoral Fellowship of the Hungarian Academy of Science, and by NKFIH (Hungarian National Research, Development
and Innovation Office) grant K-124749.

\smallskip

\noindent
G. Somlai was supported by the J\'anos Bolyai Research Fellowship of the Hungarian Academy of Sciences and the New National Excellence Program under the grant number \'UNKP-20-5-ELTE-231,  received funding from the European Research Council (ERC) under the European Union’s Horizon 2020 research and innovation programme (grant agreement No. 741420),

\smallskip

\noindent
M. Vizer was supported by the National Research, Development and Innovation Office -- NKFIH under the grants KH 130371, SNN 129364, the J\'anos Bolyai Research Fellowship of the Hungarian Academy of Sciences and the New National Excellence Program under the grant number \'UNKP-20-5-BME-45.

\appendix
\section{Appendix}

Let $N=pqr$, where $p,q,r$ are different primes and $(A,B)$ a spectral pair.
By Lemma \ref{manyprimes}, we can assume that at least two primes does not divide the cardinality of $S$.
Without loss of generality we may assume $q \nmid |S|$ and $r \nmid |S|$ and $q>r\ge 2$.
\begin{lem}\label{lemnew0723}
Let $(S, \La)$ be a spectral pair. If $\Phi_N \mid m_S$ and $gcd(|S|,N) \mid p$, then $S$ is the union of $\Z_p$-cosets.
\end{lem}

\begin{proof} It follows from $q \nmid |S|$ and $r \nmid |S|$ that  $S\cap((x+\Z_q)\cup (x+\Z_r))=\{ x\}$ for every $x \in  \Z_N$.
We proceed indirectly. Assume there is $y \in (x+\Z_p)\setminus S$. Let $H$ denote the set of points $z$ of  Hamming distance $3$ from $x$ such that $y \in C(x,z)$. Clearly, $q$ or $r$ is at least $3$ so there are point in $H$, which differ only in their $q$- or $r$-coordinate.

Note that by the 3D cube-rule $H \subseteq S$ which implies $q \mid |S|$ and $r \mid |S|$, a contradiction.
\end{proof}

Thus if $\Phi_N\mid m_S$, then by Lemma \ref{lemnew0723} and Summary \ref{summary}, it follows that $S$ is a tile. Similar argument shows the same if $\Phi_N\mid m_{\La}$. Therefore we get the following statement.

\begin{lem}\label{lem3p1}
Let $(S, \La)$ be a spectral pair of $Z_N$. If either $\Phi_N\mid m_S$ or $\Phi_N\mid m_{\La}$, then $S$ is a tile.
\end{lem}

Now suppose that $(S,\La)$ is a spectral pair such that $\Phi_N\nmid m_S$ and $\Phi_N\nmid m_{\La}$ and $q,r\nmid |S|$.
If $\Phi_{pq}\mid m_S$, then $S_{pq}$ satisfies the 2D cube rule, thus by Lemma \ref{twoprimes}, $S_{pq}$ is the union of $\Z_p$-cosets and $\Z_q$-cosets. Now we show that it is only the sum of $\Z_p$-cosets. Indeed, if $S_{pq}$ is the union of $Z_q$-cosets, then $q\mid |S|$, which is a contradiction. If two elements of $S$ having the same projection in $S_{pq}$,  $\Phi_r\mid m_{\La}$ and hence $r\mid|\La|=|S|$. Hence there is no pair of points in $S$ having the same projection. This implies $S_{pq}$ can only be the union of $\Z_p$-cosets that are not coincide, so it is the sum of $\Z_p$-cosets.

If $S_{pq}$ is a $\Z_p$-coset, then clearly $S_{pq}$ is a tile of $\Z_{pq}$ and hence $S$ tiles $\Z_N$. Now we assume that $|S|>p$. Since $\Phi_{pqr}\nmid m_{\La}$, there are no elements of $S$ of Hamming-distance 3.  Thus, if either $|S|\ge 3p$ or $p>2$  and $|S|\ge 2p$, then every element of $S$ have the same $r$-coordinate. Indeed, every element of $S$ can be reached by a path in $S_{pq}$ of length at most 3, where each consecutive elements are of Hamming-distance 2, hence their $r$-coordinate is the same. Thus $S$ is not primitive and by Summary \ref{summary} we are done. If $p=2$ and $|S|=2p=4<5$, then by Summary \ref{summary}, $S$ is not a spectral. Thus, as ${\bf T-S}(\Z_N)$ holds by \cite{Shi19}, if $N$ is square-free,  we proved the following, which was first proved in \cite{Shi19}.
\begin{Thm}\label{t3p}
Fuglede's conjecture holds on $\Z_{pqr}$, where $p,q,r$ are different primes.
\end{Thm}

\section{Appendix}
Now we state and prove the generalisation of Lemma \ref{lemprop1} as follows. We show it when $m=\frac{N}{p_ip_j}$, one can plainly modify it for other cases along the same line.
\begin{prop}\label{pprop}
Suppose $N=p_1p_2 \ldots p_k$ is a square-free integer, $m \mid N$ and let $B \subseteq \Z_N$. Suppose also that for every $l$ with $m \mid l \mid N$ we have $\Phi_l \mid m_B$. Then for every $\Z_m$-coset $\Z_m+a$ we have $\Phi_m \mid m_{(B\cap (\Z_m+a))-a}$.
\end{prop}

\begin{proof}
Lemma \ref{lemprop1} handles the case $m=\frac{N}{p_i}$ so we may assume $k \ge 2$ and $m \mid \frac{N}{p_ip_j}$, where $i \ne j$.

We prove the statement for $m=\frac{N}{p_1p_2}$. More general situation can be proved analogously. We prove that if $\Phi_N\Phi_{\frac{N}{p_1}}\Phi_{\frac{N}{p_2}}\Phi_{\frac{N}{p_1p_2}}\mid m_B$, then $\Phi_{\frac{N}{p_1p_2}}\mid m_{(B_{i,j}-b_{i,j})}$, where $B_{i,j}=\{b\in B | ~b\equiv i \pmod{p_1}, b\equiv j \pmod {p_2}\}$ and $b_{i,j}$ is a coset representative from $B_{i,j}$. For the sake of simplicity, from now on we identify $B_{i,j}-b_{i,j}$ with $B_{i,j}$.
As in the proof of Lemma \ref{lemprop1}, $\Phi_N\mid m_B$, $\Phi_{\frac{N}{p_1}}\mid m_B$,
 $\Phi_{\frac{N}{p_2}}\mid m_B$, $\Phi_{\frac{N}{p_1p_2}}\mid m_B$ is equivalent to
 \begin{equation*}
\begin{split}
0=& \sum_{i=0}^{p_1-1}\sum_{j=0}^{p_2-1}  \xi_{p_1}^{i} \xi_{p_2}^j  \sum_{t \in B_{i,j}}\xi_{p_3}^{a_3(t)}\ldots
\xi_{p_k}^{a_k(t)},\\
0=&\sum_{i=0}^{p_1-1}\sum_{j=0}^{p_2-1}  \xi_{p_2}^j \sum_{t \in B_{i,j}}\xi_{p_3}^{a_3(t)}\ldots
\xi_{p_k}^{a_k(t)},\\
0=&\sum_{i=0}^{p_1-1}\sum_{j=0}^{p_2-1}  \xi_{p_1}^i \sum_{t \in B_{i,j}}\xi_{p_3}^{a_3(t)}\ldots \xi_{p_k}^{a_k(t)},\\
0=&\sum_{i=0}^{p_1-1}\sum_{j=0}^{p_2-1}\sum_{t \in B_{i,j}}\xi_{p_3}^{a_3(t)}\ldots \xi_{p_k}^{a_k(t)},
\end{split}
\end{equation*}
where $a_h(t) \equiv t \pmod{p_h}$, for $h=3 \stb k$, respectively.
Let $t_{i,j}=\sum_{t \in B_{i,j}}\xi_{p_3}^{a_3(t)}\ldots \xi_{p_k}^{a_k(t)}$. Then we have that
 \begin{equation*}
0= \sum_{i=0}^{p_1-1}\sum_{j=0}^{p_2-1}t_{i,j} \xi_{p_1}^{i} \xi_{p_2}^j, \
0=\sum_{i=0}^{p_1-1}\sum_{j=0}^{p_2-1} t_{i,j} \xi_{p_2}^j, \
0=\sum_{i=0}^{p_1-1}\sum_{j=0}^{p_2-1} t_{i,j} \xi_{p_1}^i, \
0=\sum_{i=0}^{p_1-1}\sum_{j=0}^{p_2-1}t_{i,j}.
\end{equation*}

Let $U$ be a function from $\Z_{p_1p_2}$ to $\Q(\xi_{p_3} \stb \xi_{p_k})$ defined as $U(b_{i,j})=t_{i,j}$. Note that the minimal polynomial of $\xi_{p_1}$, $\xi_{p_2}$ and $\xi_{p_1p_2}$ over $\Q(\xi_{p_3} \stb \xi_{p_k})$ are $\Phi_{p_1}, \Phi_{p_2}$ and $\Phi_{p_1p_2}$, respectively. Then the first three equations imply that $\Phi_{p_1p_2}\Phi_{p_1}\Phi_{p_2} \mid m_U$. Since $\Phi_{p_1p_2}\Phi_{p_1}\Phi_{p_2}= \sum_{i=0}^{p_1p_2-1}x^{i}$ we have $U$ is constant. Using the last equation we obtain $t_{i,j}= 0$ for every $i \in \Z_{p_1}, j \in \Z_{p_2}$.
\end{proof}


\begin{thebibliography}{99}

\bibitem{Aten2016} C. Aten, B. Ayachi, E. Bau, D. FitzPatrick, A. Iosevich, H. Liu, A. Lott, I. MacKinnon, S. Maimon, S.
Nan, J. Pakianathan, G. Petridis, C. Rojas Mena, A. Sheikh, T. Tribone, J. Weill, C. Yu. ``Tiling sets
and spectral sets over finite fields''. \textit{J. Funct. Anal.} \textbf{273}, 8: 2547--2577, 2017.


\bibitem{CM1999} E. M. Coven, A. Meyerowitz. ``Tiling the integers with translates of one finite set''. \textit{Journal of Algebra}, \textbf{212}(1), 161--174, 1999.

\bibitem{DL2014} D. E. Dutkay, C. K. Lai. ``Some reductions of the spectral set conjecture to integers''. \textit{Mathematical Proceedings of the Cambridge Philosophical Society}, \textbf{156}(1), 123--135, 2014.

\bibitem{FFLS2015} A. Fan, S. Fan, L. Liao, R. Shi. ``Fuglede's conjecture holds in $\mathbb{Q}_p$''. \textit{Math. Annalen}, Vol. \textbf{375}, Issue 1--2, 315--341, 2019.

\bibitem{FMM2006} B. Farkas, M. Matolcsi, P. M\'ora. ``On Fuglede's conjecture and the existence of universal spectra''. \textit{Journal of Fourier Analysis and Applications}, \textbf{12}(5), 483--494, 2006.

\bibitem{FS2020} S. J. Ferguson, N. Sothanaphan. ``Fuglede's conjecture fails in 4 dimensions over odd prime fields''. \textit{Discrete Mathematics}, Vol. \textbf{343}, Issue 1, 111507, 2020.

\bibitem{F1974} B. Fuglede. ``Commuting self-adjoint partial differential operators and a group theoretic problem''. \textit{Journal of Functional Analysis}, \textbf{16}(1), 101--121, 1974.

\bibitem{IK2013} A. Iosevich, M. N. Kolountzakis. ``Periodicity of the spectrum in dimension one''. \textit{Analysis \& PDE} \textbf{6}(4), 819--827, 2013.

\bibitem{IMP2017} A. Iosevich, A. Mayeli, J. Pakianathan. ``The Fuglede conjecture holds in $\mathbb{Z}_p \times \mathbb{Z}_p$''. \textit{Analysis \& PDE}, \textbf{10}(4), 757--764, 2017.

\bibitem{KMSV2018} G. Kiss, R. D. Malikiosis, G. Somlai, M. Vizer. ``On the discrete Fuglede and Pompeiu problems''. \textit{Analysis \& PDE}, \textbf{13}(3), 765--788, 2020.

\bibitem{KS2019} G. Kiss, G. Somlai. ``Fuglede's conjecture holds on $\Z_p^2\times\Z_q$''. \textit{ArXiv preprint}, 2019: \url{https://arxiv.org/abs/1912.07114}.

\bibitem{KM2006} M. N. Kolountzakis, M. Matolcsi. ``Complex Hadamard matrices and the spectral set conjecture''. \textit{Collectanea Mathematica}, Vol. Extra, 281--291, 2006.

\bibitem{L2001} I. \L aba. ``Fuglede's conjecture for a union of two intervals''. \textit{Proceedings of the American Mathematical Society}, \textbf{129}(10), 2965--2972, 2001.

\bibitem{L2002} I. \L aba. ``The spectral set conjecture and multiplicative properties of roots of polynomials''. \textit{Journal of the London Mathematical Society}, \textbf{65}(3), 661--671, 2002.

\bibitem{LL2000} T. Y. Lam, K. H. Leung. `` On Vanishing Sums of Roots of Unity''. \textit{Journal of Algebra} \textbf{224}, 91--109, 2000.

\bibitem{LM2019} N. Lev, M. Matolcsi. ``The Fuglede conjecture for convex domains is true in all dimensions''. \textit{ArXiv preprint}, 2019: \url{https://arxiv.org/abs/1904.12262}


\bibitem{MK17} R. D. Malikiosis, M. N. Kolountzakis. ``Fuglede's conjecture on cyclic groups of order $p^nq$''. \textit{Discrete Analysis}, 2017:12, 16pp.

\bibitem{MK20} R. D. Malikiosis ``On the structure of spectral and tiling subsets of cyclic groups''. \textit{ArXiv preprint}, 2020: \url{https://arxiv.org/abs/2005.05800}.

\bibitem{Shi18} R. Shi. ``Fuglede's conjecture holds on cyclic groups $\Z_{p_1p_2p_3}$''. \textit{Discrete Analysis}, 2019:14, 14pp.

\bibitem{Shi19} R. Shi. ``Equi-distributed property and spectral set conjecture on $\Z_{p^2}\times\Z_p$''. \textit{Journal of the London Mathematical Society}, to appear.

\bibitem{So2019} G. Somlai. ``Spectral sets in $\Z_{p^2qr}$ tile''. \textit{ArXiv preprint}, 2019: \url{https://arxiv.org/abs/1907.04398}

\bibitem{T2004} T. Tao. ``Fuglede's conjecture is false in 5 and higher dimensions''. \textit{Mathematical Research Letters}, \textbf{11}(2), 251--258, 2004.

\end{thebibliography}
\end{document}